\documentclass[12pt]{amsart}

\usepackage{amsthm}

\usepackage{amssymb}
\usepackage{verbatim}
\usepackage[curve,matrix,arrow]{xy}
\usepackage{amsmath,amsfonts}
\usepackage{graphicx}
\usepackage{epsf}

\usepackage{amssymb, graphics}

\newcommand{\mathsym}[1]{{}}

\newcommand{\thmref}[1]{Theorem~\ref{#1}}
\newcommand{\propref}[1]{Proposition~\ref{#1}}
\newcommand{\lemref}[1]{Lemma~\ref{#1}}
\newcommand{\eqnref}[1]{Equation~(\ref{#1})}
\newcommand{\remref}[1]{Remark~\ref{#1}}
\newcommand{\corref}[1]{Corollary~\ref{#1}}

\newcommand{\figref}[1]{Figure~\ref{#1}}

\newcommand{\conref}[1]{Conjecture~\ref{#1}}

  {\end{list}}

\def\li{L_{i}}
\def\ri{R_{i}}
\def\ddx{{\frac{d}{dx}}}

\newtheorem{theorem}{Theorem}[section]

\newtheorem{corollary}[theorem]{Corollary}
\newtheorem{conjecture}[theorem]{Conjecture}
\newtheorem{lemma}[theorem]{Lemma}
\newtheorem{proposition}[theorem]{Proposition}

\theoremstyle{example}

\newtheorem{remark}[theorem]{Remark}

\theoremstyle{definition}

\theoremstyle{notation}

\newtheorem{example}[theorem]{Example}

\newcommand{\hj}[3]{\hat{j}_{#1}(#2,#3)}

\newcommand{\ga}{\Gamma}

\newcommand{\tg}{\tau(\Gamma)}
\newcommand{\ta}[1]{\tau(#1)}

\newcommand{\ee}[1]{E(#1)}
\newcommand{\vv}[1]{V(#1)}
\newcommand{\va}{\upsilon}

\newcommand{\pp}{p_{i}}

\newcommand{\qq}{q_{i}}

\def\<{\langle }
\def\>{\rangle }
\newcommand{\secref}[1]{\S\ref{#1}}

\def\inf{\text{inf}}

\def\mod{\text{mod}}
\def\tr{\text{tr}}
\def\diag{\text{diag}}
\newcommand{\am}{\mathrm{A}}

\newcommand{\dm}{\mathrm{D}}

\newcommand{\lm}{\mathrm{L}}
\newcommand{\plm}{\mathrm{L^+}}

\newcommand{\plp}{l_{pp}^+}

\newcommand{\plpq}{l_{pq}^+}

\def\elg{\ell (\ga)}

\begin{document}

\title[Families of Metrized Graphs With Small Tau Constants]
{Families of Metrized Graphs With Small Tau Constants}

\author{Zubeyir Cinkir}
\address{Zubeyir Cinkir\\
Zirve University\\
Faculty of Education\\
Department of Mathematics\\
Gaziantep\\
TURKEY.}
\email{zubeyirc@gmail.com}


\keywords{Metrized graph, tau constant, hexagonal net around a torus, tau lower bound conjecture}

\begin{abstract}
Baker and Rumely's tau lower bound conjecture claims that if the tau constant of a metrized graph is divided by its total length, this ratio must be bounded below by a positive constant for all metrized graphs. We construct several families of metrized graphs having small tau constants. In addition to numerical computations, we prove that the tau constants of the metrized graphs in one of these families, the hexagonal nets around a torus, asymptotically approach to $\frac{1}{108}$ which is our conjectural lower bound.
\end{abstract}

\maketitle

\section{Introduction}\label{sec introduction}

The tau constant $\tg$ of a metrized graph $\ga$ is an invariant which plays important roles in both harmonic analysis on metrized graphs \cite{BRh} and arithmetic of curves \cite{C7}. Its properties are studied in the articles \cite{C2} and \cite{C5}. Moreover, its algorithmic computation is given in \cite{C3}. The tau constant is closely related to the other metrized graph invariants studied in \cite{Zh2} and graph invariants such as edge connectivity and Kirchhoff index.

In \cite[Conjecture 14.5]{BRh}, Baker and Rumely posed a conjecture concerning the existence of a universal lower bound for $\tg$. This conjecture can be stated as follows:
\begin{conjecture}\label{conj tau lower bound1}
If $\ell(\ga)=\int_{\ga} dx$ denotes the total length of $\ga$, then we have
$$ \inf_{\ga} \frac{\tg}{\ell(\ga)}>0, $$
taking the infimum over all metrized graphs $\ga$ with $\ell(\ga) \neq 0$.
\end{conjecture}
We call this Baker and Rumely's tau lower bound conjecture.
%
%
We think that this conjecture can be refined as follows \cite{C5}:
\begin{conjecture}\label{conj lower bound3}
For all metrized graphs $\Gamma$,
$\tau(\Gamma) > \frac{1}{108} \cdot \ell(\Gamma)$.
\end{conjecture}

To disprove \conref{conj lower bound3} by finding a counterexample or to show if the conjectured lower bound is optimal, one needs to find metrized graphs with small tau constants. However, this is not an easy thing to do. This makes the lower bond conjecture a tantalizing problem.

In this article, we first review some background material in \secref{sec tau constant}. In \secref{sec symmetry}, we gather various related results about tau constant (from \cite{BRh}, \cite{BF}, \cite{C1}, \cite{C2}, \cite{C3}, \cite{C5}, \cite{C8}) to obtain necessary conditions to have small tau constants. Using the conditions we obtained in this article and interpreting our previous results, we develop a strategy to search small tau constants. In the rest of the article, we apply this strategy to construct several families of metrized graphs with small tau constants. To be more precise, we have the following contributions in this article:

In \secref{sec torus}, we prove that the tau constant of hexagonal torus $\mathrm{H}(n,m)$ approaches to $\frac{\ell(\mathrm{H}(n,m))}{108}$ whenever $n=m$ and $n \longrightarrow \infty$. This shows that the conjectured lower bound in \conref{conj lower bound3} is the best one can have if the conjecture is true. As a byproduct, we obtain the following lower and an upper bound to the Kirchhoff index of  $\mathrm{H}(n,n)$ (see \thmref{thm kirchhoff} below):
$$\frac{2n(n+1)^2(2n+3)}{3}+\frac{(n+1)^2}{3} \leq Kf(\mathrm{H}(n,n)) \leq \frac{n(n+1)^2(n+2)(n+3)}{3}+\frac{(n+1)^2}{3}.$$

In \secref{sec numerical examples}, we construct two other families of metrized graphs, and we compute their tau constants numerically by implementing the algorithm given in \cite{C3} in Matlab and Mathematica. These computations, which are improved versions of the ones given in \cite[Chapter 6]{C1}, show that metrized graphs $\mathrm{H}(n,n)$ are not the only ones having small tau constants. Moreover, the computations suggest that the tau constant of metrized graph $\ga$ in those two families approaches to $\frac{\elg}{108}$ more rapidly than the graphs $\mathrm{H}(n,n)$ as their number of vertices tends to infinity.

\section{The tau constant of a metrized graph}\label{sec tau constant}

In this section, we set the notation and recall definitions that we use in later sections.
To make this article as much self contained as possible, we revise the basic facts that we use frequently.
However, we want to make the size of this section minimal. One who seeks further information on any of the materials included in this section may consult to the provided references.

A metrized graph $\ga$ is a finite connected graph whose edges are equipped with a distinguished
parametrization \cite{RumelyBook}.

A metrized graph can have multiple edges and self-loops. For any given $p \in \ga$,
the number of directions emanating from $p$ will be called the \textit{valence} of $p$, and will be denoted by
$\va(p)$. By definition, there can be only finitely many $p \in \ga$ with $\va(p)\not=2$.

A vertex set for a metrized graph $\ga$ is a finite set of points $\vv{\ga}$ in $\ga$ which contains all the points
with $\va(p)\not=2$. It is possible to enlarge a given vertex set by adjoining additional points of valence $2$ as vertices.

Given a metrized graph $\ga$ with vertex set $\vv{\ga}$, the set of edges of $\ga$ is the set of closed line segments with end points in $\vv{\ga}$. We will denote the set of edges of $\ga$ by $\ee{\ga}$. However, we will denote the graph obtained from $\ga$ by deletion of the interior points of an edge $e_i \in \ee{\ga}$ by $\ga-e_i$.

We denote $\# (\vv{\ga})$ and $\# (\ee{\ga})$ by $v$ and $e$, respectively. We define the genus $g$ of a metrized graph as the first Betti number, i.e., $g=e-v+1$.
We denote the length of an edge $e_i \in \ee{\ga}$ by $\li$. The total length of $\ga$, which will be denoted by $\elg$, is given by $\elg=\sum_{i=1}^e\li$.

If we change each edge length in $\ga$ by multiplying  with $\frac{1}{\ell(\ga)}$, we
obtain a new graph which is called normalization of $\ga$ and denoted by $\ga^{N}$. Thus, $\ell(\ga^{N})=1$.
If $\ell(\ga)=1$, we call $\ga$ a normalized metrized graph.

The vertex connectivity of $\ga$ is the minimum number of vertices that should be deleted to make the metrized graph $\ga$ disconnected.
Similarly, the edge connectivity $\Lambda(\ga)$ of $\ga$ is the minimum number of edges that one should delete to make $\ga$ disconnected.


For any $x$, $y$, $z$ in $\ga$, the voltage function $j_z(x,y)$ on
$\ga$ is a symmetric function in $x$ and $y$, which satisfies
$j_x(x,y)=0$ and $j_x(y,y)=r(x,y)$, where $r(x,y)$ is the resistance
function on $\ga$. For each vertex set $\vv{\ga}$, $j_{z}(x,y)$ is
continuous on $\ga$ as a function of all three variables.
We have $j_z(x,y) \geq 0$ for all $x$, $y$, $z$ in $\ga$.
For proofs of these facts, see \cite{CR}, \cite[sec 1.5 and sec 6]{BRh}, \cite[Appendix]{Zh1} and \cite{BF}.

%

We consider \eqnref{eqn tau formula} (\cite[Lemma 14.4]{BRh}) below as the definition of the tau constant $\tg$ of a metrized graph $\ga$. Although, its original definition can be found in \cite[Section 14]{BRh}. For any fixed $p \in \Gamma$, we have
\begin{equation}\label{eqn tau formula}
\begin{split}
\tau(\Gamma) = \frac{1}{4} \int_\Gamma
\left( \ddx r(x,p) \right)^2 dx .
\end{split}
\end{equation}

%
One can find more detailed information on $\tg$ in articles \cite{C1}, \cite{C2}, \cite{C3} and \cite{C5}.

%
%

Let $\ga-e_i$ be a connected graph for an edge $e_i \in \ee{\ga}$ of length $\li$. Suppose $\pp$ and $\qq$ are the end points of $e_i$, and $p \in \ga-e_i$. Let $\hj{x}{y}{z}$ be the voltage function in $\ga-e_i$. Throughout this paper, we will use the following notation:
$R_{a_i,p} := \hj{\pp}{p}{\qq}$, $R_{b_i,p} := \hj{\qq}{\pp}{p}$, $R_{c_i,p} := \hj{p}{\pp}{\qq}$, and $\ri$ is the resistance
between $\pp$ and $\qq$ in $\ga-e_i$. Note that $R_{a_i,p}+R_{b_i,p}=\ri$ for each $p \in \ga$. When $\ga-e_i$ is not connected,
if $p$ belongs to the component of $\ga-e_i$ containing $\pp$ we set $R_{b_i,p}=\ri=\infty$ and $R_{a_i,p}=0$,
while if $p$ belongs to the component of $\ga-e_i$ containing $\qq$ we set $R_{a_i,p}=\ri=\infty$ and $R_{b_i,p}=0$.

Using parallel circuit reduction, we obtain
\begin{equation}\label{eqn edge resistance}
\begin{split}
r(p_i,q_i)=\frac{\li \ri}{\li + \ri}.
\end{split}
\end{equation}

By computing the integral in \eqnref{eqn tau formula}, one obtains the following formula for the tau constant:
\begin{proposition}\label{proptau}
Let $\Gamma$ be a metrized graph, and let $L_i$ be the length of the
edge $e_{i}$, for $i \in \{1,2, \dots, e\}$.
Using the notation above,
if we fix a vertex $p$ we have
\[
\ta{\ga} = \frac{1}{12} \sum_{e_i \in \ee{\ga}} \frac{\li^3}{(\li+\ri)^2}
+\frac{1}{4} \sum_{e_i \in \ga}\frac{\li(R_{a_{i},p}-R_{b_{i},p})^2}{(\li+\ri)^2}.
\]
Here, if $\ga-e_i$ is not connected, i.e. $\ri$ is infinite, the
summand corresponding to $e_i$ should be replaced by $3\li$, its limit as $\ri \longrightarrow
\infty$.
\end{proposition}
\propref{proptau} was obtained in 2003 by a REU group lead by Baker and Rumely.
Its proof can be found in \cite[Proposition 2.9]{C2}.


Chinburg and Rumely \cite[page 26]{CR} showed that
\begin{equation}\label{eqn genus}
\sum_{e_i \in \ee{\ga}}\frac{\li}{\li +\ri}=g.
\quad \text{equivalently } \sum_{e_i \in \ee{\ga}}\frac{\ri}{\li +\ri}=v-1.
\end{equation}

To have a well-defined discrete Laplacian matrix $\lm$ for a metrized
graph $\ga$, we first choose a vertex set $\vv{\ga}$ for $\ga$ in
such a way that there are no self loops, and no multiple edges
connecting any two vertices. This can be done by
enlarging the vertex set by considering additional valence two points as vertices
whenever needed. We call such a vertex set $\vv{\ga}$ adequate.
Note that if $\ga$ has no any self loops or multiple edges, then any vertex set of $\ga$ is adequate.

If distinct vertices $p$ and $q$ are the end points of an edge, we
call them adjacent vertices.

Let $\ga$ be a metrized graph with $e$ edges and an adequate vertex set
$\vv{\ga}$ containing $v$ vertices. Fix an ordering of the vertices
in $\vv{\ga}$. Let $\{L_1, L_2, \cdots, L_e\}$ be a labeling of the
edge lengths. The matrix $\am=(a_{pq})_{v \times v}$ given by
\[
a_{pq}=\begin{cases} 0, & \quad \text{if $p = q$, or $p$ and $q$ are
not adjacent}.\\
\frac{1}{L_k}, & \quad \text{if $p \not= q$, and an edge of length $L_k$ connects $p$ and $q$.}\\
\end{cases}
\]
is called the adjacency matrix of $\ga$. Let $\dm=\diag(d_{pp})$ be
the $v \times v$ diagonal matrix given by $d_{pp}=\sum_{s \in
\vv{\ga}}a_{ps}$. Then $\lm:=\dm-\am$ is called the discrete
Laplacian matrix of $\ga$. That is, $ \lm =(l_{pq})_{v \times v}$ where
\[
l_{pq}=\begin{cases} 0, & \; \, \text{if $p \not= q$, and $p$ and $q$
are not adjacent}.\\
-\frac{1}{L_k}, & \; \, \text{if $p \not= q$, and $p$ and $q$ are
connected by} \text{ an edge of length $L_k$}\\
-\sum_{s \in \vv{\ga}-\{p\}}l_{ps}, & \; \, \text{if $p=q$}
\end{cases}.
\]

The pseudo inverse $\plm$ of $\lm$, also known as the
Moore-Penrose generalized inverse, is uniquely determined by
the following properties:
\begin{align*}
i)  \quad &\lm \plm \lm  = \lm, & \qquad \qquad
iii) \quad &(\lm \plm)^{T}   = \lm \plm,
\\ ii) \quad &\plm \lm \plm  = \plm, & \qquad \qquad
iv) \quad &(\plm \lm)^{T}   = \plm \lm .
\end{align*}

We denote the trace of $\plm$ by $\tr (\plm)$.

\begin{remark}\label{rem doubcent1}
Since $\plm = (\plpq)_{v \times v}$ is doubly centered, $\sum_{p \in \vv{\ga}} \plpq =0$,
for each $q \in \vv{\ga}$. Also, $\plpq = l_{qp}^+$, for each $p$,
$q$ $\in \vv{\ga}$.
\end{remark}
The following result is known for any graph (see \cite{RB2}, \cite{RB3}, \cite[Theorem A]{D-M}).
It also holds for a metrized graph after a choice of an adequate vertex set.
\begin{lemma}\label{lem disc2}
Suppose $\ga$ is a metrized graph with the discrete Laplacian $\lm$ and the
resistance function $r(x,y)$. Let $\plm$ be a the pseudo inverse $\plm$ of $\lm$ we have
$$r(p,q)=l_{pp}^+-2l_{pq}^+ + l_{qq}^+, \quad \text{for any $p$, $q$ $\in \vv{\ga}$}.$$
\end{lemma}
One can also use $\plm$ to compute the values of voltage functions:
\begin{lemma}\cite[Lemma 3.5]{C8}\label{lem voltage1}
Let $\ga$ be a metrized graph with the discrete Laplacian $\lm$ and the voltage
function $j_x(y,z)$. Let $\plm$ be a the pseudo inverse $\plm$ of $\lm$. Then for any $p$, $q$, $s$ in $\vv{\ga}$,
$$j_p(q,s)=l_{pp}^+ - l_{pq}^+ -l_{ps}^+ +l_{qs}^+.$$
\end{lemma}
One can find more information about $\lm$ and $\plm$ in \cite[Section 3]{C3} and the references therein.

One can use \remref{rem doubcent1}, \lemref{lem disc2} and the definition of the trace to show the following equalities:
\begin{equation}\label{eqn res trace}
\begin{split}
\sum_{q \in \vv{\ga}} r(p,q) = v \cdot \plp + \tr(\plm)  \, \,  \text{  for any $p \in \vv{\ga}$, and } \, \, \sum_{p, \, q \in \vv{\ga}} r(p,q) = 2 v \cdot \tr(\plm).
\end{split}
\end{equation}
Similarly, one can show by using \remref{rem doubcent1} and \lemref{lem voltage1} that
\begin{equation}\label{eqn vol trace}
\begin{split}
\sum_{s \in \vv{\ga}}j_s(p,q)=\tr (\plm)+ v l_{p q}^+ \quad \text{for any $p, \, q \in \vv{\ga}$}.
\end{split}
\end{equation}


\section{Highly Symmetric Cases}\label{sec symmetry}
\vskip .1 in

In this section, we build an intuition about the size and behavior of the tau constant. Because, we want to understand the properties of metrized graphs that potentially have small tau constants. For this purpose, we recall all the relevant known results on tau constant. By interpreting these results, we can have educated guesses to find metrized graphs with small tau constants.

It is known that $\tg = \frac{\elg}{4}$ whenever the metrized graph $\ga$ is a tree graph (see \cite[Equation 14.3]{BRh}). Moreover, $\tg \leq \frac{1}{4}$ if $\ga$ is normalized metrized graph. In fact, if an edge $e_i$ is a bridge (an edge whose deletion disconnects the graph) of length $\li$, it makes its highest contribution to $\tg$ by $\frac{\li}{4}$. Therefore, one should consider metrized graphs without bridges to have small tau constants. In particular, this implies that we should consider metrized graphs without points of valence $1$.

If the metrized graph $\ga$ is a circle graph, we have $\tg = \frac{\elg}{12}$ (see \cite[Example 16.5]{BRh}).  Moreover, $\tg \leq \frac{\elg}{12}$ if $\ga$ is a bridgeless metrized graph (see \cite[Corrollary 5.8]{C2}), i.e. if the edge connectivity of $\ga$ is at least $2$.

We know that $\tg \geq \frac{1}{4}r(p,q)$ for any two points $p$ and $q$ in $\ga$ (see \cite[Theorem 2.27]{C2}). Thus, we want to have small maximal resistance distance in a graph to have small tau constant.

Whenever we have a vertex $p \in \vv{\ga}$ with $\va(p) \geq 3$, considering points $q \in \ga$ with $\va(q)=2$ as a vertex or not does not change $\ga$, which is called the valence property of tau constant (see \cite[Remark 2.10]{C2}). Therefore, in our search of small tau constants, we work with metrized graphs that we can keep $\vv{\ga}$ minimal by considering only the points with valence at least $3$ as vertices.

Suppose $\beta$ is the metrized graph obtained from a metrized graph $\ga$ by multiplying each edge length of $\ga$ by a constant $M$. Then $\vv{\beta}=\vv{\ga}$, $\ee{\beta}=\ee{\ga}$ and $\ell(\beta)=M \cdot \elg$. For the resistance functions $r_{\ga}(x,y)$ on $\ga$ and $r_{\beta}(x,y)$ on $\beta$, we have $r_{\beta}(x,y)=M \cdot r_{\ga}(x,y)$. Now, if we use \propref{proptau}, we obtain $\tau(\beta)= M \cdot \tau(\ga)$. This gives $\tau(\beta^N)=\frac{\tau(\beta)}{\ell(\beta)}=\frac{\tg}{\ell(\ga)}=\tau(\ga^N)$ which is one of the motivation to have \conref{conj tau lower bound1} (see \cite[Page 265]{BRh}). We call this property the scale independence of tau constant. Using this property, we can work with normalized metrized graphs in our search of small tau constants.

Now, we recall some basic facts and definitions from graph theory. A graph is called $r$-regular if $\va(p)=r$ for every vertex $p$ of the graph. For an $r$-regular graph with $e$ edges and $v$ vertices, we have
\begin{equation}\label{eqn regular}
\begin{split}
e &= \frac{r \cdot v}{2}, \, \quad \text{if the graph is $r$-regular},\\
e & \geq \frac{r \cdot v}{2}, \, \quad \text{if $\va(p) \geq r$ for every vertex $p$ in the graph}.
\end{split}
\end{equation}
A $3$-regular graph is called a cubic graph. In particular, we have $e = \frac{3 \cdot v}{2}$ for a cubic graph, and $e \geq \frac{3 \cdot v}{2}$ if $\va(p) \geq 3$ for every vertex $p$.

If $v \leq 48$, $\tg \geq \frac{\elg}{108}$ (see \cite[Theorem 6.10 part (2)]{C5}). Hence, we should consider metrized graphs having more than $48$ vertices for our purpose.

If $\ga_1 \cap \ga_2 = \{  p \}$ for any two metrized graphs $\ga_1$ and $\ga_2$, we have $\tau(\ga_1 \cup \ga_2) = \tau(\ga_1)+\tau(\ga_2)$ where the union is taken along the vertex $p$. This is called the additive property of the tau constant (see \cite[Page 15]{C2}). As a consequence of this property, we can restrict our attention to the metrized graphs with vertex connectivity at least $2$ to find smaller tau constants. In particular, we should consider metrized graphs without self loops, since the vertex connectivity is $1$ for such graphs.

The formula for the tau constant given in \propref{proptau} has two positive parts. Namely, the first part contains the term
$\sum_{e_i \in \ee{\ga}} \frac{\li^3 }{(\li+\ri)^2}$,
and the second part contains the term
$\sum_{e_i \in \ga}\frac{\li(R_{a_{i},p}-R_{b_{i},p})^2}{(\li+\ri)^2}$.
To have smaller $\tg$, one idea is to find examples of metrized graphs
for which these two terms are smaller. Now, we recall the following equality \cite[Equation 3]{C3}:
\begin{equation}\label{eqn term2b}
\begin{split}
\sum_{e_i \in \ga}\frac{\li(R_{a_{i},p}-R_{b_{i},p})^2}{(\li+\ri)^2} = \sum_{e_i \in \ga}\frac{(r(p_i,p)-r(q_i,p))^2}{\li},
\end{split}
\end{equation}
where $p_i$ and $q_i$ are the end points of the edge $e_i$.
To have smaller second term, \eqnref{eqn term2b} suggests that we need $r(p_i,p)$ and $r(q_i,p)$ be close to each other as much as possible for any choices of $e_i$ and $p$. This means that the metrized graph should have high symmetry in terms of the positioning of the vertices, and that have equal edge lengths. Moreover, this is more likely the case if we deal with a regular metrized graph.

If we have larger $\ri$'s, we will have smaller first term $\sum_{e_i \in \ee{\ga}} \frac{\li^3 }{(\li+\ri)^2}$. This suggests that we should consider metrized graphs with small edge connectivity. To put it roughly, this is because of the fact that having more parallel edges between any two points in the network means having lower effective resistance between those two points. Another result supporting this observation is the following inequality for a normalized $\ga$ \cite[Theorem 2.25]{C2}:
$$\sum_{e_i \in \ee{\ga}} \frac{\li^3 }{(\li+\ri)^2} \geq \frac{1}{(1+\sum_{e_i \in \ee{\ga}} \ri)^2}.$$
We also note that we should consider metrized graphs with larger girth to have larger $\ri$'s, where the girth is defined as the length of the shortest cycle in the metrized graph. Thus we can state the following remark:
\begin{remark}\label{rem girth}
We should consider metrized graphs with smaller edge connectivity and larger girth to have smaller first term of the tau constant.
\end{remark}

Various computations we made show that the values of these two terms from \propref{proptau} are interrelated.
That is, making one of these two terms smaller may increase the other term for any given metrized graph, and so the tau constant may not be small. For example, we can find metrized graphs with arbitrary small first terms among the family of normalized graphs $\ga(a,b,t)$ given in \cite[Example 8.2]{C2}, but then we will have very large second terms for such graphs to the point that their tau constants get closer to $\frac{1}{12}$.

Another issue to keep in mind is that the symbolic computations of tau constants can be very difficult in general. As a function of edge lengths for a metrized graph $\ga$, $\tg$ is nothing but a rational function $\frac{P}{Q}$, where $P$ and $Q$ are homogeneous polynomials in the edge lengths such that $\deg(P)=\deg(Q)+1$. Most of the time, the actual symbolic value of tau constant can be very lengthy.
However, having specific assumptions on $\ga$ can make the computation of $\tg$ easier.

Having guided by the observations above, we should consider metrized graphs $\ga$ such that
$\li=L_j$ and $\ri=R_j$ for every edges $e_i$ and $e_j$ in $\ee{\ga}$. Note that these are very strong assumptions for a metrized graph.

Whenever $\li=L_j$ and $\ri=R_j$ for every edges $e_i$ and $e_j$ in $\ee{\ga}$, $\ga$ is one of the following graphs:
\begin{itemize}
\item $\ga$ is a tree graph, i.e., $\ga$ consists of bridges only.
\item $\ga$ is the union of circle graphs along one point, i.e., $\ga$ consists of self loops only and $\ga$ has only one vertex.
\item $\ga$ does not contain any bridges or self loops.
\end{itemize}

For our purposes, we work with bridgeless metrized graphs without self loops.
\begin{theorem}\label{thm special}
Let $\ga$ be a bridgeless normalized metrized graph with no self loops such that $\li=L_j$ and $\ri=R_j$ for every edges $e_i$ and $e_j$ in $\ee{\ga}$. Then we have $\li = \frac{1}{e}$ and $\ri = \frac{v-1}{e \cdot g}$. Moreover,
\begin{align*}
\frac{\ri}{\li+\ri} &= \frac{v-1}{e},& \sum_{e_i \in \ee{\ga}} \frac{\li \ri}{\li+\ri}&=\frac{v-1}{e},&  \sum_{e_i \in \ee{\ga}} \frac{\li \ri^2}{(\li+\ri)^2}&= (\frac{v-1}{e})^2,
\\ \frac{\li}{\li+\ri}& =\frac{g}{e},&  \sum_{e_i \in \ee{\ga}} \frac{\li^2}{\li+\ri}&=\frac{g}{e},&  \sum_{e_i \in \ee{\ga}} \frac{\li^3 }{(\li+\ri)^2}&= (\frac{g}{e})^2.
\end{align*}
\end{theorem}
\begin{proof}
Since each edge length is equal and $\ga$ is normalized, $\li=\frac{1}{e}$ for every edge $e_i$ in $\ee{\ga}$. It follows from the assumptions that
$\frac{\li}{\li +\ri} = \frac{L_j}{L_j +R_j}$ and $\frac{\ri}{\li +\ri} = \frac{R_j}{L_j +R_j}$ for every edges $e_i$ and $e_j$ in $\ee{\ga}$. Now, we use \eqnref{eqn genus} to derive $\frac{\ri}{\li+\ri} = \frac{v-1}{e}$ and $\frac{\li}{\li+\ri}=\frac{g}{e}$.
Using the latter equality along with the fact that $g=e-v+1$, we obtain $\ri = \frac{v-1}{e \cdot g}$.
One can obtain the remaining equalities in the theorem using what we derived. For example,
$$\sum_{e_i \in \ee{\ga}} \frac{\li^3}{(\li+\ri)^2} = e \cdot \li (\frac{\li}{\li+\ri})^2 = (\frac{\li}{\li+\ri})^2 = (\frac{g}{e})^2 .$$
%
\end{proof}

\begin{corollary}\label{cor special}
Let $\ga$ be an $r$-regular bridgeless normalized metrized graph with no self loops such that $\li=L_j$ and $\ri=R_j$ for every edges $e_i$ and $e_j$ in $\ee{\ga}$. Then we have $\li = \frac{2}{r \cdot v}$ and $\ri = \frac{4(v-1)}{r(r-2)v^2+2rv}$. Moreover,
\begin{align*}
\frac{\ri}{\li+\ri} &= \frac{2(v-1)}{r v},& \sum_{e_i \in \ee{\ga}} \frac{\li \ri}{\li+\ri}&=\frac{2(v-1)}{r v},&
\\ \frac{\li}{\li+\ri}& =1-\frac{2(v-1)}{r v},&  \sum_{e_i \in \ee{\ga}} \frac{\li^2}{\li+\ri}&=1-\frac{2(v-1)}{r v},&
\\ \sum_{e_i \in \ee{\ga}} \frac{\li \ri^2}{(\li+\ri)^2}&= \frac{4(v-1)^2}{r^2 v^2},& \sum_{e_i \in \ee{\ga}} \frac{\li^3 }{(\li+\ri)^2}&= (1-\frac{2(v-1)}{r v})^2.
\end{align*}
\end{corollary}
\begin{proof}
The results follows from \thmref{thm special} by using the facts that $e=\frac{r v}{2}$ and $g=e-v+1$.
\end{proof}
When we work with an $r$-regular metrized graph as in \corref{cor special}, we see that the term $\sum_{e_i \in \ee{\ga}} \frac{\li^3}{(\li+\ri)^2}$ decreases as $r$ gets smaller. Whenever $r=3$ we have
\begin{equation}\label{eqn cubic case}
\begin{split}
\sum_{e_i \in \ee{\ga}} \frac{\li^3}{(\li+\ri)^2} = \frac{1}{9}(1+\frac{2}{v})^2.
\end{split}
\end{equation}
Next, we turn our attention to the second term $\sum_{e_i \in \ee{\ga}} \frac{\li \big( R_{b_i, p}-R_{a_i, p} \big)^2}{(\li + \ri)^2}$.
First, we recall the following fact, where $\tr(\plm)$ denotes the trace of $\plm$:
\begin{theorem}\cite[Theorem 4.8]{C3}\label{thm disc1}
Let $\lm$ be the discrete Laplacian matrix of size $v \times v$ for a metrized
graph $\ga$, and let $\plm$ be the pseudo inverse of $\lm$.
Suppose $\pp$ and $\qq$ are the end points of edge $e_i \in \ee{\ga}$, and $\ri$, $R_{a_i, p}$, $R_{b_i, p}$ and $\li$ are as defined before. Then we have
$$\sum_{e_i \in \ee{\ga}} \frac{\li \big( R_{b_i, p}-R_{a_i, p} \big)^2}{(\li + \ri)^2} =
\frac{4}{v} \tr(\plm) -\frac{1}{2}\sum_{q, \, s  \in \vv{\ga}} l_{qs} \big( l_{qq}^+ -
l_{ss}^+ \big)^2.$$
\end{theorem}
We know by \cite[Lemma 4.7]{C3} that
\begin{equation}\label{eqn positive part}
\begin{split}
-\frac{1}{2}\sum_{q, \, s  \in \vv{\ga}} l_{qs} \big( l_{qq}^+ -
l_{ss}^+ \big)^2=\sum_{e_i  \in \ee{\ga}} \frac{1}{\li}\big( l_{\pp
\pp}^+ - l_{\qq \qq}^+ \big)^2 \geq 0.
\end{split}
\end{equation}
In order to make the second term smaller, we want to have a metrized graph with $l_{pp}^+ =
l_{qq}^+$ for any two vertices $p$ and $q$. That is, we want $\plm$ have a constant diagonal.
The following lemma describes some sufficient conditions for this purpose, which were also used in \thmref{thm special} and \corref{cor special}.
\begin{lemma}\label{lem special}
Let $\ga$ be an $r$-regular normalized metrized graph with $\li=L_j$ and $\ri=R_j$ for every edges $e_i$ and $e_j$ in $\ee{\ga}$.
Suppose $\ga$ has no multiple edges.
Then, $l_{pp}^+ = \frac{1}{v} \tr(\plm)$ for any vertex $p$. In particular, we have
$$\sum_{e_i \in \ee{\ga}} \frac{\li \big( R_{b_i, p}-R_{a_i, p} \big)^2}{(\li + \ri)^2} =
\frac{4}{v} \tr(\plm).$$
\end{lemma}
\begin{proof}
Let $p$ be a vertex of $\ga$. Then by part $(2)$ of \cite[Lemma 4.6]{C3},
\begin{equation*}\label{eqn lpp}
\begin{split}
l_{pp}^+ &= \frac{1}{v} \tr(\plm) - \sum_{e_i  \in \ee{\ga}} \frac{\ri}{\li+\ri}( l_{p
\pp}^+ + l_{p \qq}^+ ),
\\ &= \frac{1}{v} \tr(\plm) -\frac{v-1}{e} \sum_{e_i  \in \ee{\ga}} l_{p
\pp}^+ + l_{p \qq}^+ , \quad \text{by \thmref{thm special}},
\\ &= \frac{1}{v} \tr(\plm) -\frac{v-1}{e} \sum_{q  \in \vv{\ga}} \va(q) l_{pq}^+,
\\ &= \frac{1}{v} \tr(\plm) -\frac{r(v-1)}{e} \sum_{q  \in \vv{\ga}} l_{pq}^+, \quad \text{since $\ga$ is $r$-regular}.
\end{split}
\end{equation*}
Then the equality $l_{pp}^+ = \frac{1}{v} \tr(\plm)$ follows from \remref{rem doubcent1}.
Using this equality and \thmref{thm disc1}, we obtain the second equality in the lemma.
\end{proof}
Next, we use \propref{proptau} along with the results we derived in \thmref{cor special}, \lemref{lem special} and \eqnref{eqn cubic case} to obtain the following theorem:
\begin{theorem}\label{thm special2}
Let $\ga$ be a $r$-regular normalized metrized graph with $\li=L_j$ and $\ri=R_j$ for every edges $e_i$ and $e_j$ in $\ee{\ga}$.
Suppose $\ga$ has no multiple edges, bridges and self loops. Then
$$\tg=\frac{1}{12}(1-\frac{2(v-1)}{r v})^2 + \frac{1}{v} \tr(\plm).$$
In particular, if such a metrized graph $\ga$ is cubic,
$$\tg=\frac{1}{108}(1+\frac{2}{v})^2 + \frac{1}{v} \tr(\plm).$$
\end{theorem}

\thmref{thm special2} gives specific conditions for a metrized graph $\ga$ to have simplified $\tg$ formulas. \lemref{lem special2} gives connections between such conditions and various other criteria.
\begin{lemma}\label{lem special2}
Let $\ga$ be a $r$-regular metrized graph with $\li=L_j$ for every edges $e_i$ and $e_j$ in $\ee{\ga}$. Suppose that $e_i$ has end points $p_i$ and $q_i$, and that $e_j$ has end points $p_j$ and $q_j$. Set
\begin{enumerate}
\item $\ri=R_j$ for every edges $e_i$ and $e_j$.
\item $r(p_i,q_i)=r(p_j,q_j)$ for every edges $e_i$ and $e_j$.
\item $l_{pp}^+ = \frac{1}{v} tr(\plm)$ for any vertex $p$ in $\ga$, and $l_{p_i q_i}^+ = l_{p_j q_j}^+$ for every edges $e_i$ and $e_j$. 
\item $\ga$ is $r$-regular and $r(p_i,q_i)=r(p_j,q_j)$ for every edges $e_i$ and $e_j$.
\item For any two vertices $p$ and $q$, $\sum_{s \in \vv{\ga}}r(s,p)=\sum_{s \in \vv{\ga}}r(s,q)$, and
$\sum_{s \in \vv{\ga}}j_s(p_i,q_i)=\sum_{s \in \vv{\ga}}j_s(p_j,q_j)$ for every edges $e_i$ and $e_j$.
\end{enumerate}
Then we have
$$ (4) \Rightarrow (3) \Rightarrow (2) \Leftrightarrow (1), \quad \text{and} \quad (5) \Leftrightarrow (3).$$
\end{lemma}
\begin{proof}
Since $r(p_i,q_i)=\frac{\li \ri}{\li+\ri}$, $r(p_j,q_j)=\frac{L_j R_j}{L_j+R_j}$ by \eqnref{eqn edge resistance}, $(2) \Leftrightarrow (1)$ follows from the assumption that $\li = L_j$ for every edges $e_i$ and $e_j$.

$(3) \Rightarrow (2)$ follows from \lemref{lem disc2}.

Suppose $\ga$ is $r$-regular and $r(p_i,q_i)=r(p_j,q_j)$ for every edges $e_i$ and $e_j$. Then we use $(2) \Rightarrow (1)$ and \lemref{lem special} to conclude that $l_{pp}^+ = \frac{1}{v} tr(\plm)$ for any vertex $p$ in $\ga$. On the other hand,
$r(p_i,q_i)=r(p_j,q_j)$ for any two edges $e_i$ and $e_j$ implies $l_{p_i p_i}^+ -2 l_{p_i q_i}^+ + l_{q_i q_i}^+ = l_{p_j p_j}^+ -2 l_{p_j q_j}^+ + l_{q_j q_j}^+$ for any two edges $e_i$ and $e_j$ by \lemref{lem disc2}.
This equality and the fact that $l_{pp}^+ = \frac{1}{v} tr(\plm)$ for any vertex $p$ in $\ga$ give us $l_{p_i q_i}^+ = l_{p_j q_j}^+$ for every edges $e_i$ and $e_j$. Hence, we obtain $(4) \Rightarrow (3)$.


It follows from Equations (\ref{eqn res trace}) that
$\sum_{s \in \vv{\ga}}r(s,p)=\sum_{s \in \vv{\ga}}r(s,q)$ iff $l_{p p}^+ = l_{q q}^+ $ for any two vertices $p$ and $q$.
%
%
Similarly, \eqnref{eqn vol trace} implies that
$\sum_{s \in \vv{\ga}}j_s(p_i,q_i)=\sum_{s \in \vv{\ga}}j_s(p_j,q_j)$ iff $l_{p_i q_i}^+ = l_{p_j q_j}^+ $ for every edges $e_i$ and $e_j$.
This proves $(5) \Leftrightarrow (3)$.

\end{proof}

We note that $\tg \geq \frac{\elg}{108}$ if the edge connectivity of $\ga$ is at least $6$ (see \cite[Theorem 6.10]{C5}). Moreover, we know the following upper and lower bounds to the tau constant of a $r$-regular normalized metrized graph $\ga$ with equal edge length and edge connectivity $r$ (see \cite[Theorem 6.11]{C5}):
\begin{equation}\label{eqn tau bounds}
\begin{split}
\frac{1}{12}-\frac{(v-1)(r-2)}{3 v r^2}\geq \tg \geq
\frac{1}{12}-\frac{(v-1)((r-1)v^2-5v+6)}{3r^2 v^3}.
\end{split}
\end{equation}
In particular, if $\ga$ is a $3$-regular normalized metrized graph with equal edge length and edge connectivity $3$, we have the following lower bound:
\begin{equation}\label{eqn tau bounds2}
\begin{split}
 \tg \geq \frac{1}{108}+\frac{7v^2-11v+6}{27 v^3}.
\end{split}
\end{equation}
So far in this section, we interpreted various results on $\tg$ in regards to their implications for having small tau constants. By gathering these interpretations, we expected that a metrized graph having small tau constant should satisfy various properties some of which are very strong. Then we worked with such metrized graphs and derived simplified formulas for their tau constants. However, we have not seen any example of such metrized graphs yet. Next, we give an example to show that we indeed have such special type of metrized graphs.

\begin{example}
Let $\ga$ be a normalized metrized graph which is a complete graph on $v$ vertices. Suppose $\li=L_j$ for every edges $e_i$ and $e_j$ in $\ee{\ga}$. Since $\ga$ is a complete graph with $v$ vertices, it is a regular graph with $\va(p)=v-1$ for each vertex $p$ and that it has $e=\frac{v (v-1)}{2}$ edges. Because of the symmetries of $\ga$, we have $r(p_i,q_i)=r(p_j,q_j)$ and so $\ri = R_j$ by \lemref{lem special2} for each edge $e_i$ and $e_j$. Then we use \thmref{cor special} to see $r(p_i,q_i)=\frac{\li \ri}{\li + \ri}=\frac{1}{e}\frac{2(v-1)}{r v}=\frac{4}{v^2(v-1)}$. Since any two vertices are connected by an edge, we have $r(p,q)=\frac{4}{v^2(v-1)}$ for any two vertices $p$ and $q$. Therefore, $\sum_{p, \, q \, \in \vv{\ga}} r(p,q) = (v^2-v) \frac{4}{v^2(v-1)}=\frac{4}{v}.$

On the other hand, $\sum_{p, \, q \, \in \vv{\ga}} r(p,q) = 2v \tr (L^+)$ by the second equality in Equations (\ref{eqn res trace}).
Therefore,  $2v \tr(\plm)=\frac{4}{v}$, so $\tr(\lm)=\frac{2}{v^2}$.
Finally, we use \thmref{thm special2} with $r=v-1$ to derive
$$\tg = \frac{1}{12}(1-\frac{2}{v})^2+\frac{2}{v^3}.$$
\end{example}
Our findings in the example above are consistent with \cite[Proposition 2.16]{C2}, but we used a different method in this case.

Although a complete graph with equal edge lengths have various symmetry properties that we are looking for, it has the highest regularity among the graphs without self loops and multiple edges and having $v$ vertices. Thus we don't expect to have small tau constants for such graphs as \corref{cor special} suggests. In the rest of the paper, we focus on cubic graphs with various symmetries.

\section{Hexagonal Nets Around a Torus}\label{sec torus}
\vskip .1 in

In this section, we consider the metrized graph $\ga=\mathrm{H}(n,m)$ which is the hexagonal net around a torus.
This is obtained as follows:
\begin{itemize}
  \item Make regular hexagonal tessellation of the plane.
  \item Take a region from this tessellation containing $n$ horizontal and $m$ vertical hexagonal cells as shown in \figref{fig Hnm}.
  \item Connect the vertical vertices on the boundaries by edges. Namely, there will be edges with the following pairs of end points
  $$(a_1, b_1), \, (a_2, b_2), \, (a_3, b_3), \, \cdots , \, (a_{m+1}, b_{m+1}).$$
  \item Connect the horizontal vertices on the boundaries by edges. Namely, there will be edges with the following pairs of points
  $$(a_1, d_1), \, (c_1, d_2), \, (c_2, d_3), \, \cdots , \, (c_{n-1}, d_{n}), \, (c_n, b_{m+1}).$$
\end{itemize}

\begin{figure}
\centering
\includegraphics[scale=0.9]{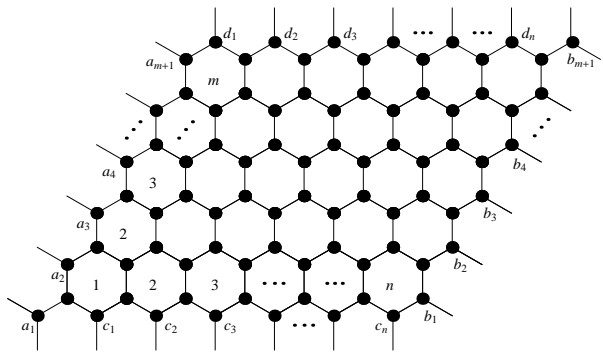} \caption{Hexagonal lattice with $n$ horizantal and $m$ vertical hexagons to form $H(n,m)$. } \label{fig Hnm}
\end{figure}

This graph is also known as hexagonal torus, honeycomb, toroidal fullerenes, toroidal polyhex or toroidal $6$-cage in the literature.
Here, we assume that each edge length in $\mathrm{H}(n,m)$ is equal to $1$ first. Then we work with its normalization.

We give the adjacency matrix of $\mathrm{H}(n,m)$, $\mathrm{A}(\mathrm{H}(n,m))$, for a better description.
We follow \cite[Section 2.4]{YZ} for the definition of the adjacency matrix.

Let $A \otimes B$ is the tensor product of the matrices $A$ and $B$.
Let $\mathrm{C}_n$ be the cycle metrized graph with $n$ vertices.
For the matrix $\mathrm{B}_n=(b_{ij})_{n \times n}$, where
\begin{equation*}\label{eqn matrix B}
b_{ij}=\begin{cases}1, &  \, \text{if } \, (i,j) \in \{ (1,2), (2,3), \cdots, (n-1,n),(n,1) \}, \\
0, &  \, \text{otherwise},
\end{cases}
\end{equation*}
we can express the adjacency matrix of $\mathrm{C}_n$ by using $\mathrm{B}_n$. Namely, $\mathrm{A}(\mathrm{C}_n)=\mathrm{B}_n+\mathrm{B}^T_n$.

For the matrix $\mathrm{F}_{2m}=(f_{ij})_{2m \times 2m}$, where
\begin{equation*}\label{eqn matrix F}
f_{ij}=\begin{cases}1, &  \, \text{if } \, (i,j) \in \{ (2,1), (4,3), \cdots, (2m,2m-1) \}, \\
0, &  \, \text{otherwise},
\end{cases}
\end{equation*}
we define the adjacency matrix of $\mathrm{H}(n,m)$ as follows:
$$\mathrm{A}(\mathrm{H}(n,m))=\mathrm{I}_{n+1} \otimes \mathrm{A}(\mathrm{C}_{2m+2})+ \mathrm{B}_{n+1} \otimes \mathrm{F}_{2m+2}+\mathrm{B}^T_{n+1} \otimes \mathrm{F}^T_{2m+2}.$$

Let $\mathrm{D}_{n,m}= \diag(3,3,\cdots, 3)$ be a diagonal matrix of size $2(n+1)(m+1) \times 2(n+1)(m+1)$. Then the discrete Laplacian matrix of $\mathrm{H}(n,m)$ is defined to be $$\mathrm{L}(\mathrm{H}(n,m))= \mathrm{A}(\mathrm{H}(n,m))-\mathrm{D}_{n,m}$$.

For example, the discrete Laplacian matrix of $\mathrm{H}(2,1)$ is as follows:

\footnotesize
$$\mathrm{L}(\mathrm{H}(2,1))=
\left(
\begin{array}{cccccccccccc}
 3 & -1 & 0 & -1 & 0 & 0 & 0 & 0 & 0 & -1 & 0 & 0 \\
 -1 & 3 & -1 & 0 & -1 & 0 & 0 & 0 & 0 & 0 & 0 & 0 \\
 0 & -1 & 3 & -1 & 0 & 0 & 0 & 0 & 0 & 0 & 0 & -1 \\
 -1 & 0 & -1 & 3 & 0 & 0 & -1 & 0 & 0 & 0 & 0 & 0 \\
 0 & -1 & 0 & 0 & 3 & -1 & 0 & -1 & 0 & 0 & 0 & 0 \\
 0 & 0 & 0 & 0 & -1 & 3 & -1 & 0 & -1 & 0 & 0 & 0 \\
 0 & 0 & 0 & -1 & 0 & -1 & 3 & -1 & 0 & 0 & 0 & 0 \\
 0 & 0 & 0 & 0 & -1 & 0 & -1 & 3 & 0 & 0 & -1 & 0 \\
 0 & 0 & 0 & 0 & 0 & -1 & 0 & 0 & 3 & -1 & 0 & -1 \\
 -1 & 0 & 0 & 0 & 0 & 0 & 0 & 0 & -1 & 3 & -1 & 0 \\
 0 & 0 & 0 & 0 & 0 & 0 & 0 & -1 & 0 & -1 & 3 & -1 \\
 0 & 0 & -1 & 0 & 0 & 0 & 0 & 0 & -1 & 0 & -1 & 3 \\
\end{array}
\right)_{12 \times 12}.$$
\normalsize

The metrized graph $\mathrm{H}(n,m)$ is a cubic metrized graph having $v=2(n+1)(m+1)$ vertices and $e=3(n+1)(m+1)$ edges, so it has $g=(n+1)(m+1)+1$.

Our motivation to consider the metrized graphs $\mathrm{H}(n,m)$ is that it satisfies the specific conditions described in \secref{sec symmetry}:
\begin{remark}\label{rem hexagon symmetry}
Because of the symmetries of the metrized graph $\mathrm{H}(n,n)$ (when $m=n$), we have the equality $r(p_i, q_i)=r(p_j, q_j)$ for each edges $e_i$ and $e_j$ in $\mathrm{H}(n,n)$.
\end{remark}

The eigenvalues of the matrix $\mathrm{L}(\mathrm{H}(n,m))$ are known to be the following values (see \cite[Section 6.2]{SW}, \cite[Section 2.4]{YZ} or \cite[Section 4]{Y})
$$\lambda_{i,j,k}=3 +(-1)^k \sqrt{3+2 \cos{\frac{2 \pi i}{n+1}}+2 \cos{\frac{2 \pi j}{m+1}}+2 \cos{\big(\frac{2 \pi i}{n+1}+\frac{2 \pi j}{m+1}\big)}},$$
where $i$, $j$ and $k$ are integers such that $0 \leq i \leq n$, $0\leq j \leq m$,  $0\leq k \leq 1$.

For general theory of toroidal fullerenes, one can see \cite[Part II]{JS}.

We know that $0$ is an eigenvalue of both $\lm$ and $\plm$ with multiplicity $1$ for a connected graph. The other eigenvalues of $\plm$
are of the form $\frac{1}{\lambda}$, where $\lambda$ is a nonzero eigenvalue of $\lm$.

Note that $\lambda_{0,0,1}=0$, $\lambda_{0,0,0}=6$ and whenever $(i,j) \neq (0,0)$ we have
$$\frac{1}{\lambda_{i,j,0}}+\frac{1}{\lambda_{i,j,1}} = \frac{3}{3-\cos{\frac{2 \pi i}{n+1}}- \cos{\frac{2 \pi j}{m+1}}- \cos{\big(\frac{2 \pi i}{n+1}+\frac{2 \pi j}{m+1}\big)}}.$$
Thus, we have the following equality for the pseudo inverse $\plm$ of $\mathrm{H}(n,m)$:
\begin{equation}\label{eqn trace L pseudo}
\begin{split}
\tr(\plm) = \frac{1}{6}+\sum_{i=0}^n\sum_{j=0}^m \frac{3}{3-\cos{\frac{2 \pi i}{n+1}}- \cos{\frac{2 \pi j}{m+1}}- \cos{\big(\frac{2 \pi i}{n+1}+\frac{2 \pi j}{m+1}\big)}},
\end{split}
\end{equation}
where $(i,j)\neq (0,0)$.

If $i=0$, we have
$$3-\cos{\frac{2 \pi i}{n}}- \cos{\frac{2 \pi j}{n}}- \cos{\big(\frac{2 \pi (i+j)}{n}\big)}=2-2 \cos{\frac{2 \pi j}{n}}=4 \sin^2{\frac{ \pi j}{n}}.$$
Now, we note the following equality (see \cite[Equation 5.2 and Corollary 5.2]{BY}, \cite[page 644]{PBM}):
\begin{equation}\label{eqn finite trig sum}
\begin{split}
\sum_{j=1}^{n-1}  \csc^2{\frac{ \pi j}{n}} = \frac{n^2-1}{3}.
\end{split}
\end{equation}
Using the fact that $\csc^2{x}=1+cot^2{x}$, we see that \eqnref{eqn finite trig sum} can be obtained from
\begin{equation}\label{eqn finite trig sum2}
\begin{split}
\sum_{j=1}^{n-1}  \cot^2{\frac{ \pi j}{n}} = \frac{(n-1)(n-2)}{3},
\end{split}
\end{equation}
which can be found in \cite[page 7]{G}, \cite[Equation 1.1]{BY} and the references therein.

Whenever $m=n$, we use \eqnref{eqn finite trig sum} to rewrite \eqnref{eqn trace L pseudo} as follows:
\begin{equation}\label{eqn trace L pseudo2}
\begin{split}
\tr(\plm) = \frac{1}{6}+\frac{n(n+2)}{2}+\sum_{i=1}^n\sum_{j=1}^n \frac{3}{3-\cos{\frac{2 \pi i}{n+1}}- \cos{\frac{2 \pi j}{n+1}}- \cos{\frac{2 \pi (i+j)}{n+1}}}.
\end{split}
\end{equation}

Next, we compute the tau constant of $\mathrm{H}^N(n,n)$:
\begin{theorem}\label{thm honeycomb}
Suppose $n \geq 3$. We have
$$\tau(\mathrm{H}^N(n-1,n-1))=\frac{n^4+11 n^2 -5}{108 n^4}+\frac{1}{2 n^4 } \sum_{i=1}^{n-1}\sum_{j=1}^{n-1} \frac{1}{3-\cos{\frac{2 \pi i}{n}}- \cos{\frac{2 \pi j}{n}}- \cos{\big(\frac{2 \pi (i+j)}{n}\big)}}.$$
\end{theorem}
\begin{proof}
The edge lengths are chosen to be $1$ for $\mathrm{H}(n-1,n-1)$, which has $v=2n^2$ vertices, $e=3n^2$ edges and genus $g=n^2-1$.
Since $\ell(\mathrm{H}(n-1,n-1))=3n^2$, we obtain $\mathrm{H}^N(n-1,n-1)$ by dividing each edge length in $\mathrm{H}(n-1,n-1)$ by $3n^2$.

Let $\plm$ be the pseudo inverse of the discrete Laplacian of the normalized metrized graph $\mathrm{H}^N(n-1,n-1)$.
Now, by using \eqnref{eqn trace L pseudo2} we derive
\begin{equation}\label{eqn trace L pseudo3}
\begin{split}
\tr(\plm) = \frac{1}{18 n^2}+\frac{n^2-1}{6 n^2}+\frac{1}{n^2}\sum_{i=1}^{n-1}\sum_{j=1}^{n-1} \frac{1}{3-\cos{\frac{2 \pi i}{n}}- \cos{\frac{2 \pi j}{n}}- \cos{\frac{2 \pi (i+j)}{n} }}.
\end{split}
\end{equation}
On the other hand, we use \remref{rem hexagon symmetry} and \thmref{thm special2} to derive
\begin{equation}\label{eqn tau formula of HT}
\begin{split}
\tau(\mathrm{H}^N(n-1,n-1))=\frac{1}{108}(1+\frac{1}{n^2})^2 + \frac{1}{2n^2} \tr(\plm).
\end{split}
\end{equation}
Hence, the proof of the theorem follows from this equality and \eqnref{eqn trace L pseudo3}.
\end{proof}

\begin{lemma}\label{lem sum inequality}
For every integer $n \geq 2$, we have
$$\frac{(n-1)^2}{6} \leq \sum_{i=1}^{n-1}\sum_{j=1}^{n-1} \frac{1}{3-\cos{\frac{2 \pi i}{n}}- \cos{\frac{2 \pi j}{n}}- \cos{\frac{2 \pi (i+j)}{n} }} \leq \frac{(n+1)(n-1)^2}{6}.
$$
\end{lemma}
\begin{proof}

\begin{equation}\label{eqn sin cos1}
\begin{split}
3-\cos{\frac{2 \pi i}{n}}- \cos{\frac{2 \pi j}{n}}- \cos{\frac{2 \pi (i+j)}{n} } = 2 \Big[ \sin^2\frac{\pi i}{n} + \sin^2 \frac{\pi j}{n}+ \sin^2 \frac{\pi (i+j)}{n}  \Big]
\end{split}
\end{equation}

and
$$
\sin^2\frac{\pi i}{n} \leq \sin^2\frac{\pi i}{n} + \sin^2 \frac{\pi j}{n}+ \sin^2 \frac{\pi (i+j)}{n} \leq 3
$$
Thus,
$$
\sum_{i=1}^{n-1}\sum_{j=1}^{n-1} \frac{1}{3} \leq \sum_{i=1}^{n-1}\sum_{j=1}^{n-1} \frac{1}{\sin^2\frac{\pi i}{n} + \sin^2 \frac{\pi j}{n}+ \sin^2 \frac{\pi (i+j)}{n}} \leq \sum_{i=1}^{n-1}\sum_{j=1}^{n-1} \frac{1}{\sin^2\frac{\pi i}{n}}
$$
Using \eqnref{eqn finite trig sum},
\begin{equation}\label{eqn sin cos2}
\begin{split}
\frac{(n-1)^2}{3} \leq  \sum_{i=1}^{n-1}\sum_{j=1}^{n-1} \frac{1}{\sin^2\frac{\pi i}{n} + \sin^2 \frac{\pi j}{n}+ \sin^2 \frac{\pi (i+j)}{n}} \leq \frac{(n+1)(n-1)^2}{3}
\end{split}
\end{equation}

Hence, the proof follows from \eqnref{eqn sin cos1} and \eqnref{eqn sin cos2}.

\end{proof}

\begin{theorem}\label{thm main}
We have
$$
\frac{n^4+20 n^2-18n+4}{108 n^4} \leq  \tau(\mathrm{H}^N(n-1,n-1))  \leq \frac{n^4+9 n^3+2 n^2-9n+4}{108 n^4}.
$$
In particular,
$\tau(\mathrm{H}^N(n,n)) \rightarrow \frac{1}{108}$ as $n\rightarrow \infty $.
\end{theorem}
\begin{proof}
The proof of the first part in the theorem follows from \thmref{thm honeycomb} and \lemref{lem sum inequality}.

Taking the limit of the each term in the first part, we see that both upper and lower bounds tend to $\frac{1}{108}$ as
$n \rightarrow \infty$. Hence, $\tau(\mathrm{H}^N(n,n)) \rightarrow \frac{1}{108}$ as $n \rightarrow \infty$.

\end{proof}

\thmref{thm main} shows that tau constants of the metrized graphs $\mathrm{H}^N(n,n)$ asymptotically approach to the conjectural lower bound $\frac{1}{108}$. Recall that if the conjectural lower bound $\frac{1}{108}$ is correct, there is no metrized graph $\beta$ with $\tau(\beta)=\frac{1}{108}$ \cite[Theorem 4.8]{C2}.

Note that the lower bound to $\tau(\mathrm{H}^N(n,n))$ given in \thmref{thm main} is better than the one given in \eqnref{eqn tau bounds}.

Now, we give approximate values of our formulas for large values of $n$. First, we observe that (see \cite[page 648]{Y})
\begin{equation}\label{eqn limit int}
\begin{split}
\lim_{n, \, m \rightarrow \infty} \frac{\tr(\plm)}{(n+1)(m+1)} &= \frac{1}{4 \pi^2} \int_{0}^{2\pi} \int_{0}^{2\pi} \frac{3}{3-\cos{x}-\cos{y}-\cos{(x+y)}} dx dy
\\ & \approx 5.4661,
\end{split}
\end{equation}
where $\plm$ is as in \eqnref{eqn trace L pseudo}. If we rewrite this equality for the pseudo inverse of $\mathrm{L}(\mathrm{H}^N(n-1,n-1))$
and use it in \eqnref{eqn tau formula of HT}, we derive the following approximation for large values of $n$:
$$
\tau(\mathrm{H}^N(n-1,n-1)) \approx \frac{1}{108}(1+\frac{1}{n^2})^2+\frac{1}{6n^2}5.4661.
$$

To have an exact formula for $\tau(\mathrm{H}^N(n-1,n-1))$, our work highlight the following questions:

What is the exact value of the following finite trigonometric sum in terms of $n$?
$$\sum_{i=1}^{n-1}\sum_{j=1}^{n-1} \frac{1}{3-\cos{\frac{2 \pi i}{n}}- \cos{\frac{2 \pi j}{n}}- \cos{\frac{2 \pi (i+j)}{n} }}.$$
Our computations using \cite{MMA} indicate that it has rational values for any given integer $n$. For example, when $2 \leq n \leq 10$ its values are as follows:

$\frac{1}{4}$, $\frac{10}{9}$, $\frac{11}{4}$, $\frac{58}{11}$, $\frac{1577}{180}$, $\frac{3812}{287}$, $\frac{529}{28}$, $\frac{419788}{16371}$, $\frac{813957}{24244}$.

F. Y. Wu \cite[Equations 2 and 11]{W} computed a similar integration (related to the sum of the eigenvalues of $\lm$ instead of $\plm$):
\begin{equation*}\label{eqn primes}
\begin{split}
\frac{1}{4 \pi^2} \int_{0}^{2\pi} \int_{0}^{2\pi} \ln{\big( 6-2 \cos{x}-2 \cos{y} -2 \cos{(x+y)} \big)} dx dy & = \frac{3 \sqrt{3}}{\pi}
(1-\frac{1}{5^2}+\frac{1}{7^2}-\frac{1}{11^2}+\frac{1}{13^2}\\
&\qquad - \frac{1}{17^2}+\frac{1}{19^2}- \cdots ).
\end{split}
\end{equation*}
Do we have a similar formula for the integral given in \eqnref{eqn limit int}?

We expect that answering these questions can help us to understand number theoretic applications of the tau constant.

We finish this section by giving lower an upper bounds to the Kirchhoff index of $\mathrm{H}(n,n)$.
\begin{theorem}\label{thm kirchhoff}
For every integer $n \geq 2$, we have
$$\frac{2n(n+1)^2(2n+3)}{3}+\frac{(n+1)^2}{3} \leq Kf(\mathrm{H}(n,n)) \leq \frac{n(n+1)^2(n+2)(n+3)}{3}+\frac{(n+1)^2}{3}.$$
\end{theorem}
\begin{proof}
Let $\plm$ be the pseudo inverse of the discrete Laplacian of $\mathrm{H}(n,n)$, and let $r(x,y)$ be the resistance function on $\mathrm{H}(n,n)$. By definition, $Kf(\mathrm{H}(n,n))=\frac{1}{2}\sum_{p, \, q, \in \vv{\mathrm{H}(n,n)}}r(p,q)$. Now, one obtains $\sum_{p, \, q, \in \vv{\mathrm{H}(n,n)}}r(p,q)=4 (n+1)^2 \tr(\plm)$ by the second equality in \eqnref{eqn res trace} and the fact that $\mathrm{H}(n,n)$ has $2(n+1)^2$ vertices.

On the other hand, we use \eqnref{eqn trace L pseudo2} and \lemref{lem sum inequality} to derive
$$ \frac{n(2n+3)}{3}+\frac{1}{6} \leq  \tr(\plm) \leq \frac{n(n+2)(n+3)}{6}+\frac{1}{6}.$$
Hence, the result follows by combining our findings.
\end{proof}


\section{Metrized Graphs With Small Tau Constants}\label{sec numerical examples}
\vskip .1 in

In this section, we first give numerical evaluations of $\tau(\mathrm{H}^N(n,m))$ for several values of $n$ and $m$.
Then we construct two other families of normalized metrized graphs. Our numerical computations suggest that the tau constants of the graphs in these families approach to $\frac{1}{108}$. Although, we don't have any theoretical proofs for these new families, the computations show that honeycomb graphs $\mathrm{H}^N(n,n)$ are not the only families of metrized graphs with small tau constants.

\textbf{Example I:} Honeycomb graphs $\mathrm{H}^N(n,m)$ with large $n$ and $m$ values, not necessarily $n=m$, have small tau constants.
This can be seen in Table \ref{table honeycomb tau}.
\begin{table}
\begin{center}
\bgroup
\def\arraystretch{1.3}
\begin{tabular}{|c|c|c|c|c|c|}
 \hline
$n  \backslash m $  & $5$ & $50$ & $100$ & $150$ & $165$ \\
 \hline
$5$  & $\frac{1}{57.21661}$ & $\frac{1}{86.28266}$ & $\frac{1}{88.80202}$ & $\frac{1}{89.67482}$ & $\frac{1}{89.83536}$ \\
 \hline
$50$  & $\frac{1}{86.28266}$ & $\frac{1}{86.28266}$ & $\frac{1}{106.93826}$ & $\frac{1}{107.22594}$ & $\frac{1}{107.27841}$ \\
 \hline
$100$  & $\frac{1}{88.80202}$ & $\frac{1}{106.93826}$ & $\frac{1}{107.44199}$ & $\frac{1}{107.61066}$ & $\frac{1}{107.64154}$\\
 \hline
$150$  & $\frac{1}{89.67482}$ & $\frac{1}{107.22594}$ & $\frac{1}{107.61066}$ & $\frac{1}{107.73206}$ & $\frac{1}{107.75424}$\\
 \hline
$165$  & $\frac{1}{89.83536}$ & $\frac{1}{107.27841}$ & $\frac{1}{107.64154}$ & $\frac{1}{107.75424}$ & $\frac{1}{107.77473}$ \\
 \hline
\end{tabular}
\egroup
\end{center}\caption{The tau constants for $\mathrm{H}^N(n-1,m-1)$ for $n$, $m$ $\in \{ 5, \, 50, \, 100, \, 150, \, 165 \}$. $\mathrm{H}^N(n-1,m-1)$ has $2 n m$ vertices. In particular, $\mathrm{H}^N(164, 164)$ has $54450$ vertices.} \label{table honeycomb tau}
\end{table}

\textbf{Example II:} Let $a$ and $b$ be two integers that are bigger than $2$. We construct a family of normalized cubic metrized graphs $\mathrm{MM}(a,b)$  with equal edge lengths and having $4ab$ vertices as follows:

First, we take two copies of circle graph $C_{a b}$ having $ab$ vertices. We label them as $A$ and $B$. Secondly, we take $a$ copies of circle graph $C_{2b}$ with $2b$ vertices. Finally, we set up edges between the circle graphs with $ab$ vertices and the circle graphs with $2b$ vertices. For each $i \in \{ 1, \, 2, \, \cdots, \, a \}$, we group the vertices of $C_i$ into sets $A_i$ and $B_i$ in an alternating way.
Then, we connect a vertex in $A_i$ to a vertex in $A$ by adding an edge. Similarly, we connect a vertex in $B_i$ to a vertex in $B$ by adding an edge. If a vertex $j \in A_i$ is connected to a vertex $k \in A$, then we add an edge between $j+2 \in A_i$ and $a+k \in A$. Similarly, if a vertex $j \in B_i$ is connected to a vertex $k \in B$, then we add an edge between $j+2 \in B_i$ and $a+k \in B$.
More precise description can be given as below:

The circle graph with label $A$ has vertices $v_1$, $v_2$, $\cdots$, $v_{ab}$ (and so the edges with the pair of end points $(v_1,v_2), \, (v_2,v_3), \, \cdots, \, (v_{ab-1},v_{ab}), \, (v_{ab},v_1)$).

The circle graph with label $B$ has vertices $v_{ab+1}$, $v_{ab+2}$, $\cdots$, $v_{2ab}$.

For each $k \in \{ 1, \, 2, \, \cdots, \, a \}$, the circle graph $C_k$ has vertices $v_{2ab+2b(k-1)+1}$, $v_{2ab+2b(k-1)+2}$, $\cdots$, $v_{2ab+2b(k-1)+2b}=v_{2b(a+k)}$.

The edges other than the ones on the circle graphs are given as follows:

For each $k \in \{ 1, \, 2, \, \cdots, \, a \}$ and for each $j \in \{ 1, \, 2, \, \cdots, \, b \}$, we have two edges with the pair of end points $(v_{2ab+2b(k-1)+2j-1},v_{a(j-1)+k})$ and $(v_{2ab+2b(k-1)+2j},v_{ab+a(j-1)+k})$, respectively. Here, the first edge connects $C_k$ and $A$, and the other edge connects $C_k$ and $B$. \figref{fig MM4and2} illustrates $MM(4,2)$.
\begin{figure}
\centering
\includegraphics[scale=0.7]{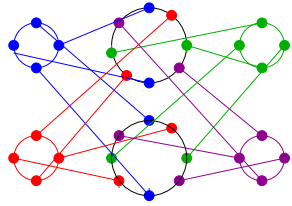} \caption{Construction of the graph MM(4,2).} \label{fig MM4and2}
\end{figure}

Because of the symmetries of $MM(a,b)$ and the fact that it is a cubic graph with equal edge lengths, we expect that its tau constant will be small.

\begin{table}
\begin{center}
\bgroup
\def\arraystretch{1.3}
\begin{tabular}{|c|c|c|c|c|c|}
 \hline
$a  \backslash b $  & $5$ & $50$ & $100$ & $110$ & $116$ \\
 \hline
 $5$ & $\frac{1}{72.89444}$ & $\frac{1}{94.18968}$ & $\frac{1}{95.74330}$ & $\frac{1}{95.88708}$ & $\frac{1}{95.96162}$ \\
 \hline
 $50$ & $\frac{1}{100.30286}$ & $\frac{1}{107.12515}$ & $\frac{1}{107.46364}$ & $\frac{1}{107.49452}$ & $\frac{1}{107.51050}$ \\
 \hline
 $100$ & $\frac{1}{102.43605}$ & $\frac{1}{107.51720}$ & $\frac{1}{107.70897}$ & $\frac{1}{107.72642}$ & $\frac{1}{107.73545}$ \\
 \hline
 $110$ & $\frac{1}{102.63448}$ & $\frac{1}{107.55209}$ & $\frac{1}{107.72935}$ & $\frac{1}{107.74546}$ & $\frac{1}{107.75379}$ \\
 \hline
  $116$ & $\frac{1}{102.73742}$ & $\frac{1}{107.57013}$ & $\frac{1}{107.73979}$ & $\frac{1}{1/107.75520}$ &  \\
 \hline
\end{tabular}
\egroup
\end{center} \caption{The tau constants for several normalized metrized graphs $MM(a,b)$ which has $4ab$ vertices. For example, $MM(116,110)$ has $51040$ vertices.} \label{table MM tau}
\end{table}

\figref{fig MMgraphs} illustrates the graphs of $MM(a,b)$ for several values of $a$ and $b$.
%
We used Mathematica \cite{MMA} to draw these graphs.


\begin{figure}
\centering
\includegraphics[scale=0.85]{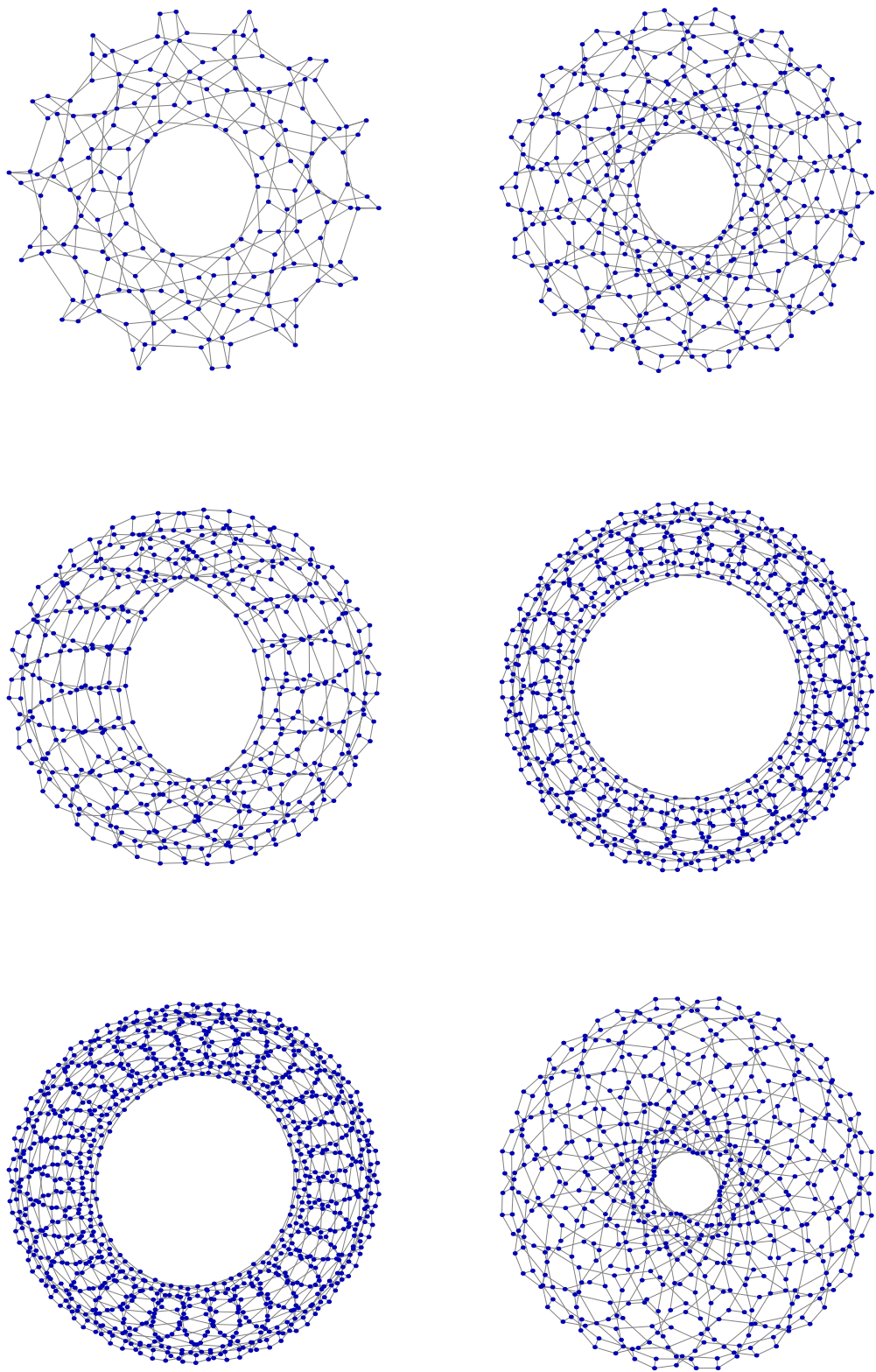} \caption{Graphs of $MM(a,b)$, where $(a,b) \in \{ (8,6), \, (12,10), \, (16,16)  \}$ on the first column and $(a,b) \in \{ (12,8), \, (12,14), \, (16,9)  \}$ on the second column.} \label{fig MMgraphs}
\end{figure}

\textbf{Example III:}
Let $a$, $b$ and $c$ be positive integers such that $b$ and $c$ are of different parity. Then we construct
a cubic normalized metrized graph $TT(a, b, c)$ with equal edge lengths as follows:

We first construct a $3$-Cayley tree containing $3 \cdot 2^a -2$ vertices. In such a graph, the number of outer vertices is $3 \cdot 2^{a-1}$ and the number of inner vertices is $k=3.2^{a-1}-2$, where the inner vertices are the ones with valence $3$. \figref{fig Cayleygraphs} shows such graphs with small $a$ values. Secondly, we add edges connecting the outer vertices.
Suppose that we label the outer vertices as $\{v_{k+1}, \, v_{k+2}, \, \cdots, \, v_{k+3 \cdot 2^{a-1}-1}, \, v_{k+3 \cdot 2^{a-1}}=v_{3 \cdot 2^a -2} \}$. Then for each odd integer $i$ with $1 \leq i \leq 3 \cdot 2^{a-1}$, we add an edge with the pair of end points
$(v_{k+i},v_{k+i+b})$, where $i+b$ is considered in $\mod$ $3 \cdot 2^{a-1}$. Similarly, for each even integer $j$ with $1 \leq j \leq 3 \cdot 2^{a-1}$, we add an edge with the pair of end points $(v_{k+j},v_{k+j+c-1})$, where $j+c-1$ is again considered in $\mod$ $3 \cdot 2^{a-1}$.
%

\begin{figure}
\centering
\includegraphics[scale=0.9]{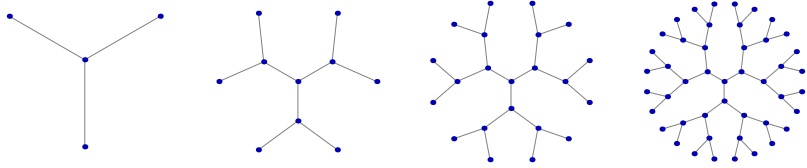} \caption{Graphs of 3-Cayley graphs when $a$ is $1$, $2$, $3$ and $4$ in order.} \label{fig Cayleygraphs}
\end{figure}

\begin{table}
\begin{center}
\bgroup
\def\arraystretch{1.3}
\begin{tabular}{|c|c|c|c|c|c|c|}
 \hline
$b  \backslash c $  & $514$ & $258$ & $130$ & $66$ & $34$ & $18$ \\
 \hline
 $2049$ & $\frac{1}{107.49402}$ & $\frac{1}{107.59561}$ & $\frac{1}{107.61874}$ & $\frac{1}{107.62882}$ & $\frac{1}{107.63435}$ & $\frac{1}{107.60193}$  \\
 \hline
 $1025$ & $\frac{1}{107.44445}$ & $\frac{1}{107.60114}$ & $\frac{1}{107.63523}$ & $\frac{1}{107.64677}$ & $\frac{1}{107.66068}$ & $\frac{1}{107.66957}$ \\
 \hline
 $513$ & $\frac{1}{106.99123}$ & $\frac{1}{107.52468}$ & $\frac{1}{107.62509}$ & $\frac{1}{107.64779}$ & $\frac{1}{107.66312}$ & $\frac{1}{107.68059}$ \\
 \hline
 $257$ & $\frac{1}{107.42122}$ & $\frac{1}{107.06822}$ & $\frac{1}{107.54748}$ & $\frac{1}{107.63829}$ & $\frac{1}{107.66162}$ & $\frac{1}{107.68262}$ \\
 \hline
  $129$ & $\frac{1}{107.59886}$ & $\frac{1}{107.44635}$ & $\frac{1}{107.10759}$ & $\frac{1}{107.56565}$ & $\frac{1}{107.63960}$ & $\frac{1}{107.65636}$  \\
 \hline
\end{tabular}
\egroup
\end{center} \caption{The tau constants for normalized metrized graphs $TT(13,b,c)$, each of which has $24574$ vertices.} \label{table TT tau}
\end{table}

\begin{table}
\begin{center}
\bgroup
\def\arraystretch{1.3}
\begin{tabular}{|c|c|c|c|c|}
 \hline
 $b  \backslash c $  & $130$ & $66$ & $34$ & $18$ \\
 \hline
 $1025$ & $\frac{1}{107.75856}$ & $\frac{1}{107.76736}$ & $\frac{1}{107.78212}$ & $\frac{1}{107.79897}$\\
 \hline
 $513$ & $\frac{1}{107.74639}$ & $\frac{1}{107.76764}$ & $\frac{1}{107.78342}$ & $\frac{1}{107.80269}$\\
 \hline
 $257$ & $\frac{1}{107.67083}$ & $\frac{1}{107.75703}$ & $\frac{1}{107.77943}$ & $\frac{1}{107.80147}$\\
 \hline
 $129$ & $\frac{1}{107.23574}$ & $\frac{1}{107.68350}$ & $\frac{1}{107.75311}$ & $\frac{1}{107.76898}$\\
 \hline
\end{tabular}
\egroup
\end{center} \caption{The tau constants for normalized metrized graphs $TT(14,b,c)$, each of which has $49150$ vertices.} \label{table TT tau2}
\end{table}


Comparing the tau values in each Tables, we note that these are the smallest tau values:

$\tau(\mathrm{H}^N(164,164))=\frac{1}{107.77473}$ and $\mathrm{H}^N(164,164)$ has $54450$ vertices, 
$\tau(MM(116,110))=\frac{1}{107.75520}$ and $MM(116,110)$ has $51040$ vertices,
$\tau(TT(13,257,18))=\frac{1}{107.68262}$ and $TT(13,257,18)$ has $24574$ vertices,
$\tau(TT(14,513,18))=\frac{1}{107.80269}$ and $TT(14,513,18)$ has $49150$ vertices.

Note that we have various graphs in the family $TT(14,b,c)$ having girth bigger than the graphs in the other families considered in this section. Thus, as suggested by \remref{rem girth}, we can expect to have metrized graphs in this family whose tau constants approaches to $\frac{1}{108}$ faster than the other ones.

Computations listed in Tables (\ref{table honeycomb tau}), (\ref{table MM tau}) and (\ref{table TT tau}) are done by using Matlab \cite{Mat}. Figures (\ref{fig Hnm}), (\ref{fig MMgraphs}), (\ref{fig Cayleygraphs}) are drawn in Mathematica \cite{MMA}.

We were able to compute the tau constant of metrized graphs having vertices more than $50,000$ and edges more than $75,000$. 
As such computations would be possible on a computer with high memory and processing speed, we used Mac Pro with processor $2 \times 2.93$ GHz $6$-core Intel Xeon ($24$ hyper-threading in total) and memory $24$ GB $1333$ MHz DDR3 to obtain these results.



\textbf{Acknowledgements:} This work is supported by The Scientific and Technological Research Council of Turkey-TUBITAK (Project No: 110T686).

\end{document}